\documentclass[12pt]{amsart}

\usepackage{amsfonts}
\usepackage{amssymb}
\usepackage{graphicx}
\usepackage{enumerate}
\usepackage{pgf,tikz}
\usepackage[letterpaper, left=2.5cm, right=2.5cm, top=2.5cm,
bottom=2.5cm,dvips]{geometry}
\usepackage{amsmath,amssymb,amsthm}
\usepackage{mathrsfs}
\usepackage{mathabx}\changenotsign
\usepackage{dsfont}
\usepackage{float}

\usepackage{xcolor}

\usepackage[backref]{hyperref}
\hypersetup{
	colorlinks,
	linkcolor={blue!60!black},
	citecolor={green!60!black},
	urlcolor={red!60!black}
}

\newtheorem{theorem}{Theorem}
\theoremstyle{plain}

\newtheorem{case}{Case}

\newtheorem{conjecture}{Conjecture}

\newtheorem{lemma}{Lemma}

\newtheorem{problem}{Problem}
\newtheorem{proposition}{Proposition}

\numberwithin{equation}{section}

\begin{document}
	\title[Square-free reducts of words]{Square-free reducts of words}

	\author{Jaros\l aw Grytczuk}
	\address{Faculty of Mathematics and Information Science, Warsaw University
		of Technology, 00-662 Warsaw, Poland}
	\email{j.grytczuk@mini.pw.edu.pl}
	\author{Szymon Stankiewicz}
	\address{Z\"{u}rich, Switzerland}
	\email{szymon.stankiewicz@gmail.com}
	\thanks{Research supported by the National Science Center of Poland, grant 2015/17/B/ST1/02660.}

\begin{abstract}
A \emph{square} is a finite non-empty word consisting of two identical adjacent blocks. A word is \emph{square-free} if it does not contain a square as a factor. In any finite word one may delete the repeated block of a square, obtaining thereby a shorter word. By repeating this process, a square-free word is eventually reached, which we call a \emph{reduct} of the original word.

How many different reducts a single word may have? It is not hard to prove that any binary word has exactly one reduct. We prove that there exist ternary words with arbitrarily many reducts. Moreover, the function counting the maximum number of reducts a ternary word of length $n$ may have grows exponentially. We also prove that over four letters, there exist words with any given number of reducts, which does not seem to be the case for ternary words. Finally, we demonstrate that the set of all finite ternary words splits into finitely many classes of \emph{related} words (one may get from one word to the other by a sequence of square reductions and factor duplications). A few open questions are posed concerning some structures on words defined with the use of the square reduction.
\end{abstract}

\maketitle

\section{Introduction}
A \emph{square} is a finite non-empty word built of two identical adjacent blocks. For instance, the word $\mathtt{hotshots}$ is a square. A word $W$ \emph{contains} a square if it can be written as $W=UXXV$ for some words $U,V$, and a non-empty word $X$. A word is \emph{square-free} if it does not contain any squares. For instance, the word $\mathtt{repetition}$ contains the square $\mathtt{titi}$, while $\mathtt{reincarnation}$ is square-free.

 It is easy to check that any binary square-free word has length at most $3$. However, there exist ternary square-free words of any length, as proved by Thue in \cite{Thue} (see \cite{BerstelThue}). This result inspired a lot of further research, leading to the birth of Combinatorics on Words, a wide discipline with lots of exciting problems, deep results, and important applications (see \cite{AlloucheShallit,BeanEM,BerstelPerrin,GrytczukDM,Lothaire,Lothaire Algebraic,NesetrilOssona}).

In this paper we consider the following problem concerning squares in words. Let $W$ be any finite word with a square, so $W=UXXV$. We may delete one of the two blocks of a square, which gives a shorter word $W'=UXV$. If $W'$ contains a square, say $W'=U'YYV'$, then we may repeat this operation of \emph{square reduction} to get another word $W''=U'YV'$. And so on, until we finally get a square-free word. Since a word may have many squares situated in different positions, this process depends on our choice of a square to be reduced in every single step. Any final word that can be obtained from a word $W$ by a sequence of square reductions is called a \emph{square-free reduct} of $W$, or just a \emph{reduct}, for short.

How many reducts a single word may have? It is not hard to demonstrate that every binary word has exactly one reduct (Proposition \ref{Proposition Binary}). However, for ternary words the situation changes dramatically. We prove (Theorem \ref{Theorem Ternary}) that for every positive integer $k$ there exists a ternary word with at least $k$ distinct reducts. Furthermore, there exist ternary words of length $n$ with exponentially many reducts (Theorem \ref{Theorem Ternary Exponential}). Also, for every integer $k\geqslant 1$ there exists a word over a $4$-letter alphabet with exactly $k$ distinct reducts (Theorem \ref{Theorem 4 letters k}). Finally, we consider a naturally defined graph on the set of ternary words, in which two words are adjacent if one can be obtained from the other by a single square reduction. We prove (Theorem \ref{Theorem Number of Components}) that this graph has finitely many connected components. In other words, we may define two words as \emph{related} if one can be obtained from the other by a finite sequence of square reductions and factor \emph{duplications} (the reverse operation to the square reduction). Then Theorem \ref{Theorem Number of Components} asserts that all ternary words form a finite number of families of mutually related words.

All these results are proved in the next section. The last section of the paoer is devoted to some open problems.

The idea of square reduction in words is not new. Indeed, the reverse operation is intensively studied in Genetics, where it is called the \emph{tandem duplication}, as the fundamental mechanism in the evolution of organisms (see \cite{Ohno},\cite{Zhang}). Naturally, some computational or language theoretic aspects are also widely studied (see \cite{AlonBFJ}, \cite{LeupoldMM}).

\section{The results}
Let $W$ be a finite word. Denote by $\mathcal{R}(W)$ the set of all reducts of $W$ and let $r(W)$ be the cardinality of $\mathcal{R}(W)$. We start with a simple result concerning binary words.

\begin{proposition}\label{Proposition Binary}
Every binary word $W$ satisfies $r(W)=1$.
\end{proposition}

\begin{proof}
There exist only six square-free binary words, namely $$\mathtt{0,1,01,10,010,101}.$$ The first two are the reducts of only two types of words, namely, $\mathtt{00\cdots0}$ and $\mathtt{11\cdots 1}$.

Now, notice that any reduct of a word $W$ must have the same initial and final letter as $W$. This completes the proof, since the remaining four square-free binary words have four different pairs consisting of the initial and the final letter.
\end{proof}

\subsection{Ternary words with many reducts} 

The above proposition is no longer true for words over a $3$-letter alphabet. For instance, the word $\mathtt{abcbabcbc}$ has two different reducts: $\mathtt{abc}$ and $\mathtt{abcbabc}$. 
We shall prove that for every $k\geqslant 1$, there exists a ternary word with at least $k$ distinct reducts. The proof is based on two lemmas involving specifically defined morphisms.

Let us define the following morphism $\varphi$ over the alphabet $\{\mathtt{a,b,c}\}$:
$$\mathtt{a}\mapsto\mathtt{abacabcbacabacbabc}$$
$$\mathtt{b}\mapsto\mathtt{abacabcbacbcacbabc}$$
$$\mathtt{c}\mapsto\mathtt{abacbcacbacabcbabc}.$$
Let us also denote $\varphi(\mathtt{a})=A$, $\varphi(\mathtt{b})=B$, and $\varphi(\mathtt{c})=C$. Recall that a morphism is \emph{square-free} if the image of every square-free word is square-free.

\begin{lemma}\label{Lemma Morphism phi}
	The morphism $\varphi$ is square-free.
\end{lemma}
\begin{proof}
By the result of Crochemore, to prove that a morphism is square-free it suffices to check its images on the set of square-free words of length at most $3$. This can be verified by computer, or by hand, provided you have a lot of patience.
\end{proof}

We write $X\leadsto Y$ to state that starting from the word $X$ one may reach the word $Y$ by a sequence of square reductions. Let $$D=\mathtt{abacabcbabcbabacabacacbcacbabcbababc}.$$

\begin{lemma}\label{Lemma D reduces to A,B}
	We have $D \leadsto A$ and $D\leadsto B$.
\end{lemma}
\begin{proof}
This can be checked by hand. Below both reductions are depicted, where blocks of the reduced squares are distinguished with red and blue, while gray part is never used.

\begin{case}
	$D\leadsto A$.
\end{case}

\center
$\color{gray}\mathtt{abacab}\color{blue}\mathtt{cbab}\color{red}\mathtt{cbab}\color{black}\mathtt{acabacacbcacbabcbaba}\color{gray}\mathtt{bc}$

$\color{gray}\mathtt{abacab}\color{black}\mathtt{c}\color{blue}\mathtt{ba}\color{red}\mathtt{ba}\color{black}\mathtt{cabacacbcacbabcbaba}\color{gray}\mathtt{bc}$

$\color{gray}\mathtt{abacab}\color{black}\mathtt{cbacaba}\color{blue}\mathtt{cacb}\color{red}\mathtt{cacb}\color{black}\mathtt{abcbaba}\color{gray}\mathtt{bc}$

$\color{gray}\mathtt{abacab}\color{black}\mathtt{cbacab}\color{blue}\mathtt{ac}\color{red}\mathtt{ac}\color{black}\mathtt{babcbaba}\color{gray}\mathtt{bc}$

$\color{gray}\mathtt{abacab}\color{black}\mathtt{cbacaba}\color{blue}\mathtt{cbab}\color{red}\mathtt{cbab}\color{black}\mathtt{a}\color{gray}\mathtt{bc}$

$\color{gray}\mathtt{abacab}\color{black}\mathtt{cbacabac}\color{blue}\mathtt{ba}\color{red}\mathtt{ba}\color{black}\mathtt{}\color{gray}\mathtt{bc}$

$\color{gray}\mathtt{abacab}\color{black}\mathtt{cbacabacba}\color{blue}\mathtt{}\color{red}\mathtt{}\color{black}\mathtt{}\color{gray}\mathtt{bc}.$

\begin{case}
$D\leadsto B$.
\end{case}

\center
$\color{gray}\mathtt{abacab}\color{blue}\mathtt{cbab}\color{red}\mathtt{cbab}\color{black}\mathtt{acabacacbcacbabcbaba}\color{gray}\mathtt{bc}$

$\color{gray}\mathtt{abacab}\color{black}\mathtt{c}\color{blue}\mathtt{ba}\color{red}\mathtt{ba}\color{black}\mathtt{cabacacbcacbabcbaba}\color{gray}\mathtt{bc}$

$\color{gray}\mathtt{abacab}\color{black}\mathtt{c}\color{blue}\mathtt{baca}\color{red}\mathtt{baca}\color{black}\mathtt{cbcacbabcbaba}\color{gray}\mathtt{bc}$

$\color{gray}\mathtt{abacab}\color{black}\mathtt{cb}\color{blue}\mathtt{ac}\color{red}\mathtt{ac}\color{black}\mathtt{bcacbabcbaba}\color{gray}\mathtt{bc}$

$\color{gray}\mathtt{abacab}\color{black}\mathtt{cbacbca}\color{blue}\mathtt{cbab}\color{red}\mathtt{cbab}\color{black}\mathtt{a}\color{gray}\mathtt{bc}$

$\color{gray}\mathtt{abacab}\color{black}\mathtt{cbacbcac}\color{blue}\mathtt{ba}\color{red}\mathtt{ba}\color{black}\mathtt{}\color{gray}\mathtt{bc}$

$\color{gray}\mathtt{abacab}\color{black}\mathtt{cbacbcacba}\color{blue}\mathtt{}\color{red}\mathtt{}\color{black}\mathtt{}\color{gray}\mathtt{bc}.$

\end{proof}

Now we are ready to prove the aforementioned result.

\begin{theorem}\label{Theorem Ternary}
	For every integer $k\geqslant 1$, there exists a ternary word $W$ with $r(W)\geqslant k$.
\end{theorem}

\begin{proof}

Consider the following three morphisms, where the words $A,B,C$ are defined above:
$$\varphi:\mathtt{a}\mapsto A, \mathtt{b}\mapsto B, \mathtt{c}\mapsto C$$
$$\varphi':\mathtt{a}\mapsto B, \mathtt{b}\mapsto A, \mathtt{c}\mapsto C$$
$$\psi:\mathtt{a}\mapsto D, \mathtt{b}\mapsto D, \mathtt{c}\mapsto C.$$
Let $S$ be any word over the alphabet $\{\mathtt{a,b,c}\}$ starting with the letter $\mathtt{a}$. We shall demonstrate that
\begin{equation}\label{Inequality r > 2r}
r(\psi(S))\geqslant 2r(S),
\end{equation}
which implies the assertion of the theorem.

By Lemma \ref{Lemma D reduces to A,B}, we have $D\leadsto A$ and $D\leadsto B$, and therefore $\psi(S) \leadsto \varphi(S)$ and $\psi(S) \leadsto \varphi'(S)$. Notice that for any two words $X,Y$ and any non-erasing morphism $\sigma$, the relation $X\leadsto Y$ implies $\sigma(X)\leadsto \sigma(Y)$. In consequence, for every square-free reduct $R$ of $S$, we have:
$$\psi(S) \leadsto \varphi(R)$$
$$\psi(S) \leadsto \varphi'(R).$$
Since both morphisms $\varphi,\varphi'$ are square-free (by Lemma \ref{Lemma Morphism phi}), both words $\varphi(R)$ and $\varphi'(R)$ are reducts of the word $\psi(S)$. Moreover, this reducts are different because the first letter of $S$, and the same of $R$, is the letter $\mathtt{a}$, while $\varphi(\mathtt{a})=A\neq B=\varphi'(\mathtt{a})$.

We have seen that one reduct $R$ of $S$ gives rise to two distinct reducts $\varphi(R),\varphi'(R)$ of $\psi(S)$. Since both morphisms $\varphi$ and $\varphi'$ are injective, this proves the inequality (\ref{Inequality r > 2r}), which completes the proof.
\end{proof}

We will now demonstrate that there are ternary words whose number of reducts is exponential in their length. This will provide also an alternative proof of Theorem \ref{Theorem Ternary}. We will make use of the following well-known fact (see \cite{Shur}).

\begin{lemma}\label{Lemma Shur}
The number of square-free words of length $n$ over a $3$-letter alphabet is at least $c^n$, for some constant $c>1$.
\end{lemma}

Let $f_k(n)$ be defined as the maximum value of $r(W)$ over all words of length $n$ over alphabet of size $k$.

\begin{theorem}\label{Theorem Ternary Exponential}
	There exists a constant $\alpha>1$ such that $f_3(n)\geqslant \alpha^n$.
\end{theorem}

\begin{proof}Let $A,B,C$, and $D$ be the same four words used in the proof of Theorem \ref{Theorem Ternary}. Consider a word $W_m=CDDDCDDD\cdots CDDD=(CDDD)^m$, with $m\geqslant 1$. By Lemma \ref{Lemma D reduces to A,B}, each cube $DDD$ can be reduced to any of the words $A,B,AB,BA,ABA,BAB$. So, the word $W_m$ can be reduced to any square-free word over alphabet $\{A,B,C\}$ having $m$ letters $C$ and starting with $CA$ or $CB$. The number of such words is at least $c^m$, by Lemma \ref{Lemma Shur}. By Lemma \ref{Lemma Morphism phi}, each such reduct is also a square-free word over alphabet $\{\mathtt{a,b,c}\}$. Thus, $W_m$ has at least $c^m$ distinct reducts. Since the length of $W_m$ is at most $36m$, the assertion of the theorem follows for $\alpha=c^{1/36}$.
\end{proof}

The best possible value of the constant $c$ from Lemma \ref{Lemma Shur} is roughly $1.3$ (see \cite{Shur}), so, for the second constant we get roughly $\alpha\approx 1.0073$. 

\subsection{Words over four letters}

We first prove the following result.

\begin{theorem}\label{Theorem 4 letters k}
For every integer $k\geqslant 1$, there exists a word over a $4$-letter alphabet with exactly $k$ distinct reducts.
\end{theorem}

\begin{proof}
Let us fix the $4$-letter alphabet as $\{\mathtt{a,b,x,y}\}$. Take the word $$F=\mathtt {xabaxababx}$$ and notice that it has exactly two distinct reducts, namely, $P=\mathtt{xabx}$ and $Q=\mathtt{xabaxabx}=Q'P$. Let $W_\infty$ be any infinite square-free word over the alphabet $\{\mathtt{a,b,y}\}$ starting with the letter $\mathtt {y}$. Let $W_1,W_2,\dots$ be any sequence of prefixes of the word $W_\infty$ with strictly growing lengths. We may assume that each word $W_i$ also ends with the letter $\mathtt {y}$. Define the sequence of words $S_i$ by
$$S_i=FW_1FW_2\cdots FW_i,$$
for all $i\geqslant 1$.

We are going to prove that $r(S_i)=i+1$ for every $i\geqslant 1$. More precisely, we will demonstrate that the set of reducts of each word $S_i$ consists of the following elements:
\begin{equation}\label{Reducts of S_i}
\mathcal{R}(S_i)=\{PW_i,QW_i\} \cup \{PW_1QW_i,PW_2QW_i,\dots,PW_{i-1}QW_i\}.
\end{equation}

First notice that all the above words are indeed square-free. It is also not hard to check that all can be obtained from $S_i$ by square reductions. For example, since $W_{j+1}=W_jY_j$, we may write
\begin{equation}
S_i=FW_1FW_2\cdots FW_i=FW_1FW_1Y_1FW_1Y_1Y_2\cdots FW_1Y_1Y_2\cdots Y_{i-1},
\end{equation}
which clearly reduces to $FW_i$, and in turn to $PW_i$ and $QW_i$. To get the words from the second part of the union (\ref{Reducts of S_i}), just write $S_i=S_{i-1}FW_i$ and apply the induction.

To prove the statement of (\ref{Reducts of S_i}) we apply the induction. The first word $S_1=FW_1$ can be written as $S_1=F'\mathtt{xy}W_1'$ and we see that the word $\mathtt{xy}$ occurs only once in $S_1$. Hence, there are no squares in $S_1$ except those contained in $F$ and therefore $\mathcal{R}(S_1)=\{PW_1,QW_1\}$, as asserted in (\ref{Reducts of S_i}). Similarly, the word $S_2=FW_1FW_2$ can be written as $S_2=FW_1''\mathtt{yx}F''W_2$ and we see that the word $\mathtt{yx}$ occurs only once in $S_2$. Hence, there are no other squares in $S_2$ beyond those contained entirely in $FW_1$ and $FW_2$, and the square $FW_1FW_1$. Since $P$ is a suffix of $Q$, the only possible new reduct is $PW_1QW_2$.

Now, we shall analyze possible square reductions in the word $S_3=FW_1FW_2FW_3$. The only essential case is when a square overlaps all three words $W_1,W_2,W_3$. Notice that it cannot entirely contain any two of them since they all have different lengths. So, the only possibility is that $S_3=FW_1'XFYXFYW_3'$, where $W_1=W_1'X$, $W_2=YX$ and $W_3=YW_3'$. Deleting the block $XFY$ gives $FW_1'XFYW_3'=FW_1FW_3$. Thus, we see that such reductions cannot produce anything new beyond words presented in (\ref{Reducts of S_i}).

For the general case, notice that any square reduction in $S_i$ either takes place in some factor $FW_jFW_{j+1}$, or must delete one word $W_j$ (remember that no square may contain entirely any two of the words $W_j$). In the later case we obtain a word with the smaller number of words $W_j$, which completes the  proof by the induction.
\end{proof}

Notice that the construction of words $S_i$ from the above proof actually shows that for every $k\geqslant 1$, there exist infinitely many words over a $4$-letter alphabet, each with exactly $k$ distinct reducts.

Our next theorem gives an exponential lower bound for the function $f_4(n)$ (better than the one following trivially from Theorem \ref{Theorem Ternary Exponential}). We will need the following simple lemma.

\begin{lemma}\label{Lemma Form of Squares}
Let $W=a_1a_2\dots a_n$ be any square-free word, where each $a_i$ is a single letter. Let $V=a_1^{k_1}a_2^{k_2}\dots a_n^{k_n}$, where each $k_i$ is a positive integer. Then every square in $V$ is of the form $x^{2k}$, where $x=a_i$ for some $i=1,2,\dots,n$.
\end{lemma}

\begin{proof}
Let $XX$ be a square in the word $V$ and assume that $X$ contains at least two distinct letters. Let $X_0$ be the word obtained from $X$ by reducing all single letter repetitions. We have two possibilities: the first and the last letter of $X_0$ are the same, or they are different. In the later case, the square $X_0X_0$ is contained in $W$, a contradiction. In the former case, the square $X_0'X_0'$ is contained in $W$, where $X_0'$ is obtained from $X_0$ by erasing the last letter. This proves the lemma.
\end{proof}

\begin{theorem}\label{Theorem 4-letters exponential}
Let $U$ be a word over alphabet $\{\mathtt{a,b,x}\}$ starting and ending with the letter $\mathtt{x}$. Let $\alpha=r(U)^{\frac{1}{\lvert U \rvert+5}}$. Then there exists a constant $c>0$ such that $f_4(n)\geqslant c\alpha^n$, for all $n\in \mathbb{N}$.
\end{theorem}
\begin{proof}
	Let $S$ be any infinite square-free word over alphabet $\{\mathtt{a,b,y}\}$ starting with the letter $\mathtt{y}$. Let $T$ be a word obtained form $S$ by duplicating every occurrence of the letter $\mathtt{y}$ in $S$, except the first one. Hence, the word $T$ can be written uniquely as $T=T_1T_2T_3\cdots$, where each factor $T_i$ starts and ends with the letter $\mathtt{y}$, and these are the only occurrences of this letter in $T_i$. Finally, let us define $V_j=UT_1UT_2\cdots UT_j$, for each $j\geqslant 1$.
	
	We claim that $r(V_j)=r(U)^j$, for every $j\geqslant 1$. Indeed, there are exactly $j$ copies of $U$ in $V_j$ and each copy can be reduced to $r(U)$ distinct reducts. Moreover, there are no other squares in $V_j$ than those entirely contained in copies of $U$. To see this, suppose that $XX$ is a square intersecting some factors $T_i$. Then the word obtained from $XX$ by erasing all letters from the copies of $U$ must have the form of a square, say $X'X'$. This square is contained in the word $T$ and has at least two distinct letters, which contradicts Lemma \ref{Lemma Form of Squares}.
	
	To finish the proof, just notice that each word $T_i$ has length at most $5$, thus, $\lvert V_j \rvert \leqslant j(\lvert U \rvert +5)$. In consequence, we have
	$$r(V_j)=r(U)^j\geqslant (r(U)^{\frac{1}{\lvert U \rvert +5}})^{\lvert V_j \rvert}=\alpha^{\lvert V_j \rvert}.$$
	This completes the proof (with $c=1$) for $n=\lvert V_j \rvert$. For the intermediate values of $n\in (\lvert V_j \rvert, \lvert V_{j+1} \rvert)$, notice that $f_4(n)$ is at least $r(V_j)$, as we may append sufficiently many letters $\mathtt{y}$ to the word $V_j$ to get a word with the same number of reducts as $V_j$. Since ${\lvert V_{i+1} \rvert}-{\lvert V_i \rvert}\leqslant {\lvert U \rvert}+5$, we may take $c=\alpha^{-(\lvert U \rvert+5)}$ to get the asserted inequality for all $n\in \mathbb{N}$.
\end{proof}
One may check that the word $U=\mathtt{xabaxababxbabx}$ satisfies $r(U)=4$ and $\lvert U \rvert = 14$. Hence, in the above theorem we may take $\alpha = 4^{\frac{1}{19}}\approx1.075$.

\subsection{Kinship of words}

For a fixed integer $k\geqslant 1$, let us consider the directed graph ${\overrightarrow{\mathcal{G}}}_k$ defined on the set of all finite non-empty words over an alphabet of size $k$, with $U$ joined to $W$ by an oriented edge $(U,W)$ whenever $W$ can be obtained from $U$ in a single square reduction. Let ${\mathcal{G}}_k$ denote the underlying \emph{undirected} graph obtained from ${\overrightarrow{\mathcal{G}}}_k$ by ignoring orientations of the edges. 

One may ask many natural questions about these graphs. Perhaps the most fundamental concerns connectedness of ${\mathcal{G}}_k$. By Proposition \ref{Proposition Binary}, it is easy to see that ${\mathcal{G}}_2$ has exactly six different components. Indeed, there are exactly six square-free words over a binary alphabet and every binary word reduces to exactly one of them. As there are infinitely many square-free words over a $3$-letter alphabet, one may expect that the number of components in ${\mathcal{G}}_3$ is also infinite. However, we shall demonstrate that this is not the case.

\begin{theorem}\label{Theorem Number of Components}
	The graph ${\mathcal{G}}_3$ has finitely many connected components.
\end{theorem}

\begin{proof} First we observe that every square-free word over alphabet $\{\mathtt{a,b,c}\}$ of length exactly $9$ contains as a factor one of the following five words (up to a permutation of the alphabet):
	\begin{equation}
X_1=\mathtt{abcabac}, X_2=\mathtt{abcacba}, X_3=\mathtt{abcbabc}, X_4=\mathtt{abcbacab}, X_5=\mathtt{abcbacb}.
	\end{equation}
These words are reducts of the following five words, respectively:

$S_1=\mathtt{abcbabcbcacbcacabacabcbacabcabacacbcabacac}$

$S_2=\mathtt{abcbabcbcacbcabacbcabcbacbcabcacbcbabcba}$

$S_3=\mathtt{abcbabcbcacbcacabacabcbabcbc}$

$S_4=\mathtt{abcbabcbcacbcacabacabcbacabcabacacbcacbacab}$

$S_5=\mathtt{abcbabcbcacbcabacbcabcbacbcabcacbabcacbcb}$.

Now, it can be checked that each word $S_i$ has another reduct $Y_i$, which is strictly shorter than the corresponding word $X_i$:
\begin{equation}
Y_1=\mathtt{abcbac}, Y_2=\mathtt{abcba}, Y_3=\mathtt{abc}, Y_4=\mathtt{abcab}, Y_5=\mathtt{abcacb}.
\end{equation}
It follows that for every ternary word $W$ there is a path in the graph ${\mathcal{G}}_3$ to some square-free word of length at most $8$. Indeed, suppose that $W$ has no reducts of length at most $8$. If $R=AX_iB$ is any reduct of $W$, then we may go up to the word $R'=AS_iB$, and then go down to the shorter word $R''=AY_iB$. After a finite number of such operations we must reach a word of length smaller than $9$.

We see that, in consequence, the number of connected components in the graph ${\mathcal{G}}_3$ does not exceed the number of square-free ternary words of length at most $8$, which completes the proof.
\end{proof}

\section{Discussion}
Let us conclude the paper with some observations and open problems. First, let us point on a strange phenomenon concerning possible numbers of reducts of ternary words. We checked by a computer that for every $m\leqslant79$, there exist ternary words $W$ with $|\mathcal{R}(W)|=m$. However, it seems that there is no ternary word having exactly $80$ reducts (see Table \ref{Table Reducts}).

\begin{problem}
Is there a ternary word $W$ with $|\mathcal{R}(W)|=80$?
\end{problem}
Other missing values up to $120$ are:
$$95,97,101,102,104,105,107,117,119,$$
and it looks like these numbers become more and more frequent. For instance, we have not found words having the number of reducts in the interval from $182$ up to $192$.

\begin{conjecture}\label{Conjecture Missing Values} There exist infinitely many positive integers $m$ such that no ternary word have exactly $m$ distinct reducts.
\end{conjecture}
Of course, we have examined only words of some bounded length. So, it may potentially happen that every missing value is eventually realized by some very long word. Let us stress, however, that actually we do not have a proof that any single value is really missing.
\begin{table}[H]

	\begin{tabular}{||c c c||} 
		\hline
		$W$ & $|W|$ & $|\mathcal{R}(W)|$  \\ [0.5ex] 
		\hline\hline
		$\mathtt{abcbabcbc}$ & $9$ & $2$  \\ 
		\hline
		$\mathtt{abcbabcbcacbcabcb}$ & $17$ & $3$  \\
		\hline
		$\mathtt{abcbabcbcacbca}$ & $14$ & $4$ \\  
	    \hline
	    $\mathtt{abcbabcbcacbcabacbcabcb}$ & $23$ & $5$ \\ 
	    \hline
	    ... & ... & ... \\ 
	    \hline
	    $\mathtt{abcbabcbcacbcacabacabcbabcbcacbcabacababcbabcacbabcabc}$ & $54$ & $79$ \\ 
	    \hline
	    $\mathtt{abcbabcbcacbcacabacabcbacabcabacacbcacbabcbcacbca}$ & $49$ & $81$ \\ 
	    \hline
	\end{tabular}
\caption{Words with growing number of reducts}
\label{Table Reducts}
\end{table}

Many natural question can be asked about the structure of graphs ${\overrightarrow{\mathcal{G}}}_k$. For instance, how long is a shortest directed path from a word $W$ to some of its reducts? This is the same as the minimum number of single square reductions needed to turn $W$ into a square-free word. Let us denote this parameter by $d(W)$, and let $d_k(n)$ denote the maximum value of $d(W)$ over all words of length $n$ over a $k$-letter alphabet.

These parameters for $k=2$ were introduced and studied by Alon, Bruck, Farnoud, and Jain in \cite{AlonBFJ} under the name of the \emph{duplication distance}. Among many interesting results, it is proved in \cite{AlonBFJ} that the limit $\lim_{n\rightarrow\infty}\frac{d_2(n)}{n}$ exists and satisfies the following inequalities:
$$0.045\leqslant \lim_{n\rightarrow\infty}\frac{d_2(n)}{n} \leqslant 0.4.$$It would be nice to investigate this quantity for larger alphabets, in particular for $k=3$.

\begin{conjecture}\label{Conjecture The Limit}
	For every $k\geqslant 2$, the limit $\lim_{n\rightarrow\infty}\frac{d_k(n)}{n}$ exists. 
\end{conjecture}

It is not hard to demonstrate that if any of the above limits exist, then it must be strictly positive.

Another parameter one could study in this topic is the out-degree of a vertex in the graph ${\overrightarrow{\mathcal{G}}}_k$. This is the same as the number of words that can be obtained from a given word $W$ in a single square reduction. Let us denote this quantity by $\deg^+(W)$, and let $\deg_k^+(n)$ be the maximum value of $\deg^+(W)$ over all words of length $n$ over a $k$-letter alphabet.

\begin{conjecture}\label{Conjecture Max Degree}
	For every $k\geqslant 2$, we have $\deg_k^+(n)\leqslant n$.
\end{conjecture}

This problem resembles one of the most famous conjectures in Combinatorics on Words, made by Fraenkel and Simpson \cite{FraenkelSimpson}, stating that every word of length $n$ contains at most $n$ \emph{distinct} squares. Let us point however, that reduction of the same square in a word may give different results, as well as reduction of distinct squares may give the same word.

Finally, it would be nice to know if Theorem \ref{Theorem Number of Components} holds for every $k\geqslant 2$. The following Conjecture seems plausible.

\begin{conjecture}\label{Conjecture Number of Components} For every $k\geqslant 2$, the graph $\mathcal{G}_k$ has finitely many connected components.
\end{conjecture}

\end{document}